\documentclass[12pt]{article}

\usepackage{fullpage}
\usepackage{amssymb}
\usepackage{amsmath}
\usepackage{amsthm}
\usepackage{hyperref}
\usepackage{graphicx}
\usepackage{url}
\newtheorem{theorem}{Theorem}

\newtheorem{definition}{Definition}

\begin{document}

\title{Record values in appending and prepending bitstrings to runs of binary digits}
\date{October 5, 2018}
\author{Chai Wah Wu\\ IBM Research AI\\IBM T. J. Watson Research Center\\ P. O. Box 218, Yorktown Heights, New York 10598, USA\\e-mail: chaiwahwu@ieee.org}
\maketitle

\begin{abstract}
In this short note, we show a simple characterization of integers that reach records for a sequence described by adding binary strings to runs of 1's and 0's in a binary representation. In particular, we show that this set does not depend on the added strings as long as they are nonempty and of the same length.
\end{abstract}

\section{Introduction}\label{sec:intro}
The sequence A175046 in the Online Encyclopedia of Integer Sequences (OEIS) \cite{oeis} is described as follows. For each integer $n$, take the runs of $1$'s in a binary representation of $n$ and append $1$ to them and take runs of $0$'s and append a $0$ to them, and convert the resulting binary string back into an integer $a(n)$. For example, for $n = 89$, the binary representation is $1011001$, and appending each runs of $0$'s and $1$'s with $0$ and $1$ respectively resulted in $110011100011$ which implies $a(89) = 3299$. Neil Sloane conjectured and Maximilian Hasler proved that $a(n)\leq \frac{9n^2+12n}{5}$, with equality if and only if $n = \frac{2}{3}(4^k-1)$, i.e. $n$ is $1010\dots 10$ in binary.  Neil Sloane also conjectured that the record values of $a(n)$\footnote{i.e. values $a(n)$ such that $a(m) < a(n)$ for all $m < n$.} (OEIS sequences A319422, A319424) are described as having a binary representation that is either a
alternating sequence of $11$ and $00$ or having a single $00$ replaced by $000$. 
In this short note we show that this conjecture is true. In fact, we show that the set of the indices $n$ of the record values is independent of the binary strings that are appended (or prepended) to the runs.

\section{Notation} \label{sec:count}
For an integer $n$, let $b(n)$ and $c(n)$ be the number of runs of $1$'s and $0$'s and the number of bits in the binary representation of $n$ respectively. 
For an integer $n$, let $B(n)$ denote the bitstring of its binary representation.
\begin{definition}
Given two binary strings $d_0$ and $d_1$, let $f(n,d_0,d_1)$ be defined as the number whose binary representation is constructed by appending $d_1$ to each run of $1$'s and $d_0$ to each run of $0$'s. Similarly $g(n,d_0,d_1)$ is defined by prepending $d_i$ rather than appending. By abuse of notation, we also apply this to the binary representation of $n$, i.e.  $f(B(n),d_0,d_1) = B(f(n,d_0,d_1))$. \label{def:one}
\end{definition}
We will omit the arguments $d_0$ and $d_1$ in $f(n,d_0,d_1)$ and $g(n,d_0,d_1)$ when they are clear from context. 
Consider the special case of $f$ when $d_0 = 0$ and $d_1 = 1$ (OEIS sequence A156064). A left inverse of $f$ (OEIS sequence A318921) in this case is described in \cite{lenormand:2003}.
It is easy to see that $b(f(n,0,1)) = b(n)$, $c(f(n,0,1)) = b(n)+c(n)$, $g(n,0,1) = f(n,0,1)$ and $g(n,1,0) = \left\lfloor f(n,0,1)/2\right\rfloor$.

If $d_0$ and $d_1$ are both length $k$ bitstrings, then 
\begin{eqnarray}
c(f(n,d_0,d_1)) &=& b(n)k+c(n) \label{eqn:one} \\
c(g(n,d_0,d_1)) &=& b(n)k + c(n) - l \label{eqn:two}\\
B(g(n,d_1,d_0)) &=& d_0B\left(\left\lfloor f(n,d_0,d_1)/2^k\right\rfloor\right) \label{eqn:three}
\end{eqnarray}
where $l$ is the number of leading zeros of $d_1$. A consequence of Eq. (\ref{eqn:three}) is that if $d_0$ does not contain $1$'s, then 
$g(n,d_1,d_0) = \left\lfloor f(n,d_0,d_1)/2^k\right\rfloor$.

For a binary bitstring $x$, we denote $\overline{x}$ as the bitstring where the $0$'s are change to $1$'s and vice verse.
In other words $\overline{x}$ is the $1$'s complement of $x$.

\section{A characterization of record values of $f$ and $g$}
\begin{theorem} \label{thm:one}
Let $S_f$ be the set $\{n: \forall m < n \quad f(m,d_0,d_1) < f(n, d_0, d_1) \}$ where $d_1$ and $d_0$ are nonempty bitstrings of the same length.
Then $n \in S_f$ if and only if the binary representation of $n$ is either an alternating sequence of $0$'s and $1$'s or an alternating sequence of $0$'s and $1$'s where exactly one of the $0$ is replaced with $00$.
\end{theorem}
\begin{proof}
Let $T$ be the set of binary sequences of alternating $0$'s and $1$'s plus sequences of alternating $0$'s and $1$'s where exactly one of the $0$ is replaced with $00$. First we show that if both $d_0$ and $d_1$ are nonempty bitstrings, then  $S_f \subset T$.
Consider $n\in S_f$ such that $B(n)$ contains $2$ consecutive $1$'s, i.e.
$B(n) = x11y$. Note that $x$ could be the empty bitstring. 
Then $m$ with $B(m) = x10\overline{y}$ clearly satisfy $m<n$.
Both $m$ and $n$ has the same number of bits, but $b(m) = b(n)+1$, so $c(f(m)) > c(f(n))$ and thus $f(m) > f(n)$. This implies that sequences in $S_f$ are {\em Fibbinary} numbers (OEIS sequence A003714).

Next suppose  that $B(n)$ contains $4$ or more consecutive $0$'s, i.e. $B(n) = x10000y$. Consider $m < n$ with $B(m)  = x01010\overline{y}$. Note that because of the above, $n\in S_f$ implies that $x$ is either the empty string or ends in $0$. Then $b(m) = b(n)+2$ and
$c(m) \geq c(n)-1$ and by Eq. (\ref{eqn:one}), this implies that $f(m) > f(n)$. 
Now suppose that $B(n)$ contains $3$ consecutive $0$'s, i.e. $B(n) = x1000y$ where $y$ does not start with $0$, i.e. $y$ is empty or starts with $1$. Consider $m < n$ with $B(m)  = x0101\overline{y}$. It is easy to see that $b(m) = b(n) + 1$. If $x$ is not the empty bitstring, then $c(m) = c(n)$ so again $c(f(m)) > c(f(n))$ and  $f(m) > f(n)$. If $x$ is the empty bitstring, then $c(m) = c(n)-1$ 
and $B(f(n)) = 1d_1000d_0f(y)$, $B(f(m)) = 1d_10d_01d_1f(\overline{y})$. If $d_0$ and $d_1$ have length $k > 1$, then $c(f(m)) > c(f(n))$. If $d_0$ and $d_1$ are both single-bit strings, then
$c(f(m)) = c(f(n))$ and comparing their initial bits shows that $f(m) > f(n)$.
Thus elements of $S_f$ cannot contain $3$ or more $0$'s in its binary representation, i.e. $S_f$ is a subset of the terms of OEIS sequence A003796, in particular it is a subset of the terms in OEIS sequence A086638.

Now suppose $n$ is such that $B(n)$ has two occurrences of $00$'s, i.e. $B(n) = x100y00z$. By the discussion above, $y$ must start and end with a $1$. Consider $m < n$ with $B(m) = x010y01\overline{z}$. It can easily been shown that $b(m) = b(n) + 1$, so again $f(m) > f(n)$ if $x$ is not the empty string.
If $x$ is the empty string then $B(f(n)) = 1d_100d_0f(y)f(00z) = 1d_100d_0...$ and $B(f(m)) = 1d_10d_0f(y)0d_0f(1\overline{z}) = 1d_10d_01...$. Again, $f(m) > f(n)$ if $d_0$ and $d_1$ are of length $k > 1$.
If $d_0$ and $d_1$ are single bits, $c(f(n)) = c(f(m))$ and comparing the initial bits shows that $f(m) > f(n)$.
This shows that $S_f \subset T$.

Next we show that $T\subset S_f$.
Consider an integer $n$ such that $B(n)$ is an alternating sequence of $0$'s and $1$'s. This implies that $b(n) = c(n)$. Consider an integer $m < n$. Clearly, $c(m)\leq c(n)$. Since the alternating sequence of $0$'s and $1$'s is the only sequence such that $b(n) = c(n)$, this means that $b(m) < b(n)$ and thus $f(m) < f(n)$.
Thus $n \in S_f$.
Next suppose that $n$ is an integer such that  $B(n) = x100y$, where $x$ is either empty or is an alternating sequences of $1$'s and $0$'s ending in $0$ and $y$ is either empty or an alternating sequence of $1$'s and $0$'s. Note that $b(n) = c(n) - 1$ and $B(f(n)) = f(x)1d_100d_0f(y)$. Consider $m < n$.  It is clear that $B(m)$ cannot be the alternating string of $1$'s and $0$'s of length $c(n)$. Thus $b(m) < c(n)$, i.e. $b(m) \leq b(n)$. Suppose $c(m) < c(n)$, then by Eq. (\ref{eqn:one})
$c(f(m)) < c(f(n))$, i.e. $f(m) < f(n)$.

Suppose $c(m) = c(n)$. If $b(m) < b(n)$, then again $f(m) < f(n)$ by Eq. (\ref{eqn:one}), so we can assume that $b(m) = b(n)$.
This implies that $B(m)$ is also of the form $x'100y'$, where $x'$ is either empty or is an alternating sequences of $1$'s and $0$'s ending in $0$ and $y'$ is either empty or an alternating sequence of $1$'s and $0$'s. Since $m < n$, the only possibility is that the $00$ of $B(m)$ is to the left of the $00$ in $B(n)$.
This $B(m) = z1001r$ and $B(n) = z1010r'$. Thus $f(m) = f(z)1d_100d_0f(1r)$ and $f(n) = f(z)1d_10d_01d_1f(0r')$. This implies that $f(n)$ and $f(m)$ has the same initial bits
followed by the bits $0d_0$ for $f(m)$ and $d_01$ for $f(n)$ and this combined with the fact that $c(f(n)) = c(f(m))$ implies that $f(n) > f(m)$.
This shows that $T\subset S_f$ and concludes the proof. 
\end{proof}

As a result of Theorem \ref{thm:one}
for $d_0 = 0$ and $d_1=1$, the values of $f(n)$ for $n \in S_f$ is exactly the numbers whose binary representation is an alternating sequences of $11$ and $00$ with at most one of the $00$ replaced with $000$, proving the (second) conjecture stated in Section. \ref{sec:intro}.

Theorem \ref{thm:one} is also valid for the function $g$ that prepends $d_i$ rather than appends $d_i$ to runs of $0$'s and $1$'s.
\begin{theorem} \label{thm:two}
Let $S_g$ be the set $\{n: \forall m < n \quad g(m,d_0,d_1) < f(n, d_0, d_1) \}$ where $d_1$ and $d_0$ are nonempty bitstrings of the same length.
Then $n \in S_g$ if and only if the binary representation of $n$ is either an alternating sequence of $0$'s and $1$'s or an alternating sequence of $0$'s and $1$'s where exactly one of the $0$ is replaced with $00$.
\end{theorem}
\begin{proof}
The proof is virtually identical to the proof of Theorem \ref{thm:one} and uses Eq. (\ref{eqn:two}) instead of Eq. (\ref{eqn:one}). The main difference is in some of the cases considered. First, for the case where $B(n) = 1000y$ and $B(m) = 101\overline{y}$ where $y$ is the empty string or starts with $1$.
In this case, $B(g(n)) = d_11d_0000g(y)$ and $B(m) = d_11d_00d_11g(\overline{y})$. If $d_0$ and $d_1$ are of length $k>1$, then $c(g(m)) < c(g(n))$ and $g(m) > g(n)$. For $d_0$ and $d_1$ both a single bit, comparing the initial bits of $g(n)$ and $g(m)$ shows that $g(m) > g(n)$. 
Second, for $B(n) = 100y00z$ and $B(m) = 10y01\overline{z}$, 
$B(g(n)) = d_11d_000g(y)d_000... = d_11d_000d_1...  $ and $B(g(m)) = d_11d_00d_11...$. For $d_0$ and $d_1$ a single-bit string, $c(g(n)) = c(g(m))$ and comparing the initial bits of $g(n)$ and $g(m)$ shows that $g(m) > g(n)$.
Third, for the  case where $B(m) = z1001r$ and $B(n) = z1010r'$. In this case,
 $g(m) = g(z)d_11d_000g(1r) =  g(z)d_11d_000d_1\cdots$ and $g(n) = g(z)d_11d_00d_11g(0r')  = g(z)d_11d_00d_11d_0\cdots$. This implies that $g(n)$ and $g(m)$ has the same initial bits followed by $0d_1$ for $g(m)$ and $d_11$ for $g(n)$ which implies that $g(n) > g(m)$.
\end{proof} 

Theorem 1 and 2 imply that when $d_0$ and $d_1$ are nonempty and of the same length, $S_g = S_f$ and is equal to the terms in OEIS sequence A319423.

\end{document}